\documentclass[11pt]{amsart}
\usepackage{amsmath, amssymb, amsthm,color}
\usepackage{amssymb,longtable}
\usepackage{amsfonts}
\usepackage{setspace}
\usepackage{amsmath} 
\usepackage{hyperref} 
\usepackage{graphicx}
\usepackage{latexsym}
\usepackage{array}
\usepackage{amssymb}
\theoremstyle{plain}

\numberwithin{equation}{section}
\newtheorem{thm}{Theorem}[section]
\newtheorem{prop}[thm]{Proposition}
\newtheorem{lemm}[thm]{Lemma}
\newtheorem{cor}[thm]{Corollary}

\theoremstyle{remark}

\newcommand{\bq}{\begin{equation}}
\newcommand{\eq}{\end{equation}}
\begin{document}

\title{Sharp $L^p$ estimates for singular transport equations}
\author{Tarek M. Elgindi}

\address{Department of Mathematics, Princeton University.}
\email{tme2@math.princeton.edu}
\date{June 18, 2014}
\maketitle

\begin{abstract}
We prove that $L^p$ estimates for a singular transport equation are sharp by building what we call a \emph{cascading solution};  the equation we consider studies the combined effect of multiplying by a bounded function and application of the Hilbert transform. Along the way we prove an invariance result for the Hilbert transform which could be of independent interest. Finally, we give an example of a bounded and \emph{incompressible} velocity field $u$ for which the equation:

$$\partial_t f +u\cdot\nabla f= H(f)$$ develops sharp $L^p$ growth. 
The equations we study are relevant, as models, in the study of fluid equations as well as in general relativity.  

\end{abstract}

\setstretch{1.2}

\section{Background}

The interaction between transport and non-locality is found in a variety of physical applications such as fluid mechanics, general relativity, biological aggregation, etc.  A particular equation we may consider is the 2-D Euler equations of incompressible flow:

$$\partial_t u+(u\cdot\nabla)u=-\nabla p,$$
$$\text{div}(u)=0,$$

where the pressure term, $p$, satisfies:

$$\Delta p=\sum_{i\not=j}\partial_{i}u_{j}\partial_{j}u_{i}.$$ Clearly, the operator taking $u$ to $p$ is non-local. In this system, $u$ is the (vector) velocity field of the fluid and it is advected by itself under no external forces other than internal pressure. 

One of the main difficulties here is the interaction between the transport operator $\partial_{t} +(u\cdot\nabla)$ and the non-local pressure term.  In two-dimensions, it is very useful to get rid of the pressure term by passing to the equation for the vorticity, $\omega:=curl(u)=\partial_y u_1-\partial_x u_2:$

$$\omega_{t}+u\cdot\nabla\omega=0.$$

One of the important features of the vorticity equation is the conservation of all the $L^p$ norms of $\omega:$

$$|\omega(t)|=|\omega_{0}|_{L^p}\,\, \forall 1<p\leq \infty, \,\, \forall t>0.$$

This fact (when $p=\infty$) allows one to prove global existence and uniqueness of smooth solutions to the incompressible Euler equations in two dimensions (\cite{Hol},\cite{Wol}). The simple structure of the vorticity equation allows us, more or less, to bypass the problem of non-locality which is introduced by the pressure term. On the other hand,  as was shown in \cite{EM}, the conservation of the $L^p$ norms for the vorticity is very unstable with respect to very mild perturbations. Indeed, consider:

\begin{equation} \partial_{t}\omega+u\cdot\nabla\omega=R(\omega),\end{equation}
\begin{equation} curl(u)=\omega, \end{equation}
\begin{equation} div(u)=0, \end{equation}
where $R$ is some non-local degree-zero operator such as a Riesz transform. This equation has become of great interest due to the fact that it can be seen as an inescapable model for many problems arising in fluid mechanics. See, for example \cite{CV}, \cite{ER}, \cite{LM}, \cite{HR}, \cite{H}. 

The question of global well-posedness is wide open for (1.1)-(1.3). 
As was shown in \cite{EM}, an $L^\infty$ estimate is impossible if $R$ is unbounded on $L^\infty$, which is often the case. As a first approach towards proving well-posedness one might try to get good $L^p$ estimates on $\omega$.  A naive approach to doing the $L^p$ estimate yields the following estimate:

\begin{equation} |\omega(t)|_{L^p}\leq e^{pt} |\omega_0|_{L^p}.\end{equation}

The fact that the $p$ shows up in the exponent is quite alarming because it indicates that the only estimates we could get on $\omega$ would be in spaces which are subcritical with respect to the scaling of the Euler equations. Indeed, to prove well-posedness for (1.1)-(1.3) one would need an estimate on $\omega$ in a space that scales like $L^\infty$ (such as $BMO$) \cite{BKM}. Thus, even though we have an estimate on $\omega$ in $L^p$ for all $p$, the speed at which the $L^p$ norms grow as $p$ becomes large becomes enormously important. If it were possible to strengthen estimate (1.4) to, for example:

$$ |\omega(t)|_{L^p}\leq p C(t) |\omega_0|_{L^p} $$ then one might be in a position to prove global well-posedness for smooth solutions of (1.1)-(1.3).  One might be skeptical that such a growth as in (1.4) is possible due to the fact that a transport equation with a divergence free velocity field does not increase the $L^p$ norm of the initial data so one should expect that the $L^p$ norms should not grow faster than the growth found in the linear equation $\omega_t=R(\omega)$ (whose solutions grow in $L^p$, at worst, like $p C(t)$).

Upon some deeper thought, one \emph{could} imagine a scenario where the velocity field and the operator $R$ work together to produce uncontrollable growth. In this work we investigate the combined \emph{linear} effect of transport and the application of a singular integral operator and show that this is indeed possible.

\subsection{The main results}

We begin by studying the following linear equation: 

$$\partial_{t}f(x,t)=H(af)(x,t),$$
$$f(x,0)=f_{0}(x),$$

where $a$ is a given $L^\infty$ function and $H$ is the Hilbert transform on the real line. 

Let us note that when $a\equiv 1$, one can solve the evolution equation exactly: $$u(x,t)=u_0(x)cos(t)+H(u_0)sin(t)$$ so that $$|u|_{L^p}\lesssim (1+pt)|u_0|_{L^p}$$ locally in time for $p\geq 2$. 

Hence, for the case $a\equiv 1$, the $L^p$ norms of the equation can grow at most linearly in $p.$

Now we turn to the case of general $a$. By Trotter's formula, one can interpret solutions of this equation as a continuous piecing together of solutions of 

$$\partial_t f=H(f),\,\,\, \text{and} \,\,\, \partial_t f=af.$$

Each of these equations, individually, is harmless.  Specifically, solutions of $f_{t}=H(f)$ satisfy the bound $|f|_{L^p}\leq C(t)p|f_0|_{L^p}$ and solutions of $f_t=af$ satisfy $|f|_{L^p}\leq C(a,t)|f_0|_{L^p}.$ Together, however, they are able to produce very non-trivial behavior. In the coming theorem, we will see that putting the two effects together--multiplication on the Fourier side by a bounded function and multiplication on the physical side by a bounded function--can produce $L^p$ growth of the order of $e^{pt}!$

\begin{thm}

Consider the following evolution equation for $(x,t)\in\mathbb{R}\times\mathbb{R}^+:$

\begin{equation} \partial_t f =H(\chi_{[0,1]}f)\end{equation}
\begin{equation} f(x,0)=f_0(x)\end{equation}

Then, if $f_0\in L^p,$ $f(t)$ remains in $L^p$ for all time with the following bound:

\begin{equation} |f(t)|_{L^p}\leq |f_{0}|_{L^p} e^{cpt}\end{equation} for some universal $c>0.$

Furthermore, this bound is sharp: there exists $f_0\in L^1\cap L^\infty$ such that:

$$|f(t)|_{L^p}\geq |f_0|_{L^p}e^{cpt}$$ for some $c>0.$ In particular, (1.1)-(1.2) is ill-posed in all spaces at the scaling of $L^\infty,$ including BMO.

\end{thm}

A corollary of the proof of Theorem 1.1, which will be discussed in Section 3, is:

\begin{thm}
Consider the evolution equation$$\partial_{t} f+ \big(0,\chi_{[0,1]}(x)\big)\cdot\nabla_{x,y} f=H_{x}f$$ with $$f(t=0,x,y)=e^{y}\chi_{[0,1]}(x).$$

Then, $$f(t,x,y)=e^y M(t,x)$$ with $$|M(t)|_{L^p}\geq e^{ct^2 p}.$$

\end{thm}

On the other hand, the following two theorems are true:

\begin{thm}
Let $u$ be a divergence-free and Lipschitz function. Then, if $f$ solves \begin{equation} \partial_t f+u\cdot\nabla f=0\end{equation} with $f_0\in BMO$, then
$$|f(t)|_{BMO}\leq C(|u|_{Lip},t)|f_0|_{BMO},$$ and 

$$|f(t)|_{L^p}\leq C(|u|_{Lip},t)p|f_0|_{BMO}.$$

\end{thm}

We remark that theorem 1.3 is sharp in the sense that if $u$ is not Lipschitz then such a theorem is, in general, not true (see \cite{BK} and \cite{BEK} ).

Theorem 1.1 seems to indicate that there exists $u\in Lip$ with $div(u)\not=0$ so that solutions of the transport equation above actually grow like $e^{pt}.$ It seems like one might be able to achieve this by refining the techniques of this paper. 

\begin{thm}
Let $a$ be a Dini continuous function of a real variable. Let $f$ solve the evolution equation $$\partial_{t} f(t,x)=H(af)(x,t)$$
$$f(x,0)=f_0(x)$$ with $f_0\in L^p.$

Then, 

$$|f|_{L^p}\lesssim_{a,t}p|f_0|_{L^p}.$$

\end{thm}

\emph{Remark:} The proof of Theorem 1.4 is a simple application of the fact that if $a$ is dini continuous then the commutator $$[a,H]f:= aH(f)-H(af)$$ maps $L^\infty$ to itself. 

We remark that this equation was studied in two dimensions by Klainerman and Rodnianski \cite{KR}. Simply, they wanted conditions on $a$ such that the above "cascade" cannot happen in the $L^1$ case.
Their assumptions are: take $a\in B^{1}_{2,1}$ and solve the linear equation $$\partial_{t}f(t,x)=M(af)(t,x),$$ assuming that $M$ is a convolution operator with \emph{bounded} and \emph{smooth} symbol (excluding the Hilbert transform).  

Then,  $|f|_{L^1}\leq C(a)|f_0|_{L^1} Log^+|f_0|_{L^\infty}.$

The point here is that only one $Log$ term  is lost just like only one $p$ is lost in the $L^p$ estimate when $a$ is Dini continuous. It is likely that if $a$ were only bounded, then one could build a cascading solution in the $L^1$ case just as we do in the $L^p$ case. Finally, whether the result of \cite{KR} can be extended to the case where $M$ is the Hilbert transform is not clear to us. 

\subsection{The idea of the proof of Theorem 1.1}
The proof of Theorem 1.1 will involve several steps. First we will derive an exact formula (up to a recursion) for $H(\chi_{[0,1]}\Big(Log\big| \frac{x}{x-1}\big|\Big)^j)$. It turns out that one can expand this quantity as a sum of powers of $\Big(Log\big| \frac{x}{x-1}\big|\Big)^l.$ We will then use these formulas to provide an exact solution for (1.5)-(1.6) through a series expansion. The type of solution we will look for is:

$f(t,x)=\sum_{j=0}^{\infty}(\alpha_{j}(t)+\beta_{j}(t)\chi_{[0,1]})\Big(Log\big| \frac{x}{x-1}\big|\Big)^j$ with $\alpha_{j}(0)=\beta_{j}(0)=0$ for all $j$ except for $\beta_{0}(0)$ which is taken to be $1.$

The coefficients $\alpha_{j}(t)$ and $\beta_{j}(t)$ will solve an infinite system of ODE's. We do not attempt to solve the system or the recursion formulas exactly but we prove lower bounds on the coefficients based on a bootstrap argument. We will prove that $\alpha_{j}(t)+\beta_{j}(t)\geq\frac{t^j}{c(j!)^2}.$ This will be enough to conclude. The following figure illustrates the first steps of the "cascade"; we start with a bounded function and end with a function with a singularity of the order or $Log^2.$

\vspace{10mm}

\includegraphics{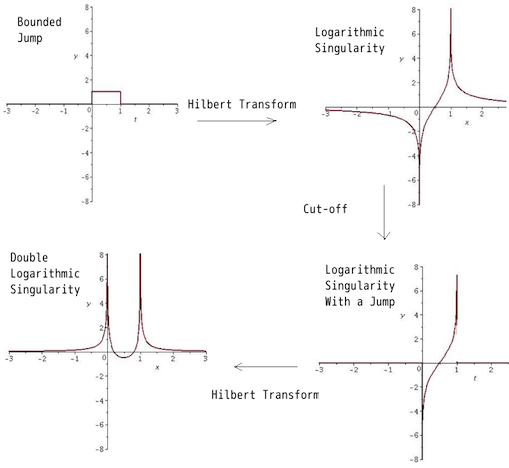}

\section{The proof of Theorem 1.1}
The proof will be based on three lemmas. 

\begin{lemm}

Denote by $a=:\chi_{[0,1]}$ the characteristic function of the unit interval. 
Then, by explicit calculation, $$H(a)=\frac{1}{\pi}Log\big|\frac{x}{x-1}\big |.$$

Furthermore, the following holds:

\begin{equation} H(aH(a)^k)= \frac{1}{k+1}H(a)^{k+1}+\sum_{j=0}^{k-1}(b^{k}_{j}+c^{k}_{j}a)H(a)^j.\end{equation}

$$b_{j}^k=c_{j}^k=0, \,\, \text{when} j-k \,\, \text{is even}$$

Moreover, $b_{j}^k$ and $c_{j}^k$ satisfy the following bounds:

$$|b_{j}^k|, |c_{j}^{k}|\leq k^{k-j}.$$

\end{lemm}

\emph{Remark:} One can check numerically that such an upper bound on the coefficients $b_j^k$ and $c_j^k$ seems to be sharp up to constants; moreover, it seems that all the $b_{j}^k$ and $c_{j}^k$ are negative. This is the difficulty we will face in proving the growth after this because it is possible that there be growth in "high modes" (logarithms with high powers) but then the system could collapse on itself by many "low modes" which are going in the opposite direction. In fact, one can already see this in the figure above.

\begin{proof}
The proof will proceed by deriving a recursion formula for the coefficients $b_{j}^k$ and $c_{j}^k.$ Just to get an idea of what is going on, we will explicitly calculate the cases $k=1$ and $k=2$.

We will be making use of the Tricomi identity \cite{King}:

$$H(fg)=gH(f)+fH(g) +H(H(f)H(g))$$ which is valid for any $f$ and $g$ that belong to, say, $L^2(\mathbb{R})$ \cite{Okada}.

Now we explicitly deal with the cases of $k=1$ and $k=2$.

$$H(aH(a))=H(a)^2-a^2-H(aH(a)).$$

Therefore,

$$H(aH(a))=\frac{1}{2}H(a)^2-\frac{1}{2}a^2.$$ In particular, the theorem holds for $k=1$.

Now for $k=2.$ We need to compute $H(aH(a)^2).$ 

We have that $$H(aH(a))=\frac{1}{2}H(a)^2-\frac{1}{2}a^2.$$

Now multiply by $a$ and take $H$ of both sides and we see:

$$H(aH(aH(a)))=\frac{1}{2} H(aH(a)^2)-\frac{1}{2}H(a).$$

Now use the Tricomi identity on the left-hand side and we see:

$$H(a)H(aH(a))-a^2H(a)-H(aH(a)^2)=\frac{1}{2}H(aH(a)^2)-\frac{1}{2}H(a)$$

Now the situation is clear:

$$H(a) (\frac{1}{2}H(a)^2-\frac{1}{2}a)-aH(a)+\frac{1}{2}H(a)=\frac{3}{2}H(aH(a)^2).$$

Therefore, 

$$H(aH(a)^2)=\frac{1}{3}H(a)^3-aH(a)+\frac{1}{3}H(a).$$

Thus the lemma is true for $k=2.$

Now we redo this procedure for the case $k=j$ assuming the theorem is true for $k\leq j-1.$ 

Indeed,

We know that  $$H(aH(a)^{j-1})=\frac{1}{j} H(a)^j +\sum_{n=0}^{j-2}(b_{n}^{j-1}+c_{n}^{j-1}a)H(a)^n$$

Now multiply by $a$ and apply $H:$

$$H(aH(aH(a)^{j-1}))=H(a\frac{1}{j} H(a)^j) +H\Big(\sum_{n=0}^{j-2}(b_{n}^{j-1}+c_{n}^{j-1})aH(a)^n\Big).$$

Now we use the product rule for $H$ on the left hand side. 

$$H(a)H(aH(a)^{j-1})-aH(a)^{j-1}-H(aH(a)^j)=H(a\frac{1}{j} H(a)^j) +H\Big(\sum_{n=0}^{j-2}(b_{n}^{j-1}+c_{n}^{j-1})aH(a)^n\Big).$$

Now, by the induction hypothesis, we know the formula for $$H(aH(a)^{n}), $$ when $n\leq j-1.$

Therefore $$H(a) \Big(  \frac{1}{j} H(a)^j +\sum_{n=0}^{j-2}(b_{n}^{j-1}+c_{n}^{j-1}a)H(a)^n\Big)-aH(a)^{j-1}$$ $$-\sum_{n=0}^{j-2}(b_{n}^{j-1}+c_{n}^{j-1})\big(\frac{1}{n+1}H(a)^{n+1}+\sum_{l=0}^{n-1}(b_{l}^n+c_{l}^na)H(a)^l \big)=\frac{j+1}{j} H(aH(a)^j) $$

Therefore, 

$$H(aH(a)^j)=\frac{1}{j+1}H(a)^{j+1}+\frac{j}{j+1}R$$ with $$R=\sum_{n=0}^{j-2}(b_{n}^{j-1}+c_{n}^{j-1}a)H(a)^{n+1}-aH(a)^{j-1}$$ $$-\sum_{n=0}^{j-2}\frac{H(a)^{n+1}}{n+1}(b_{n}^{j-1}+c_{n}^{j-1})-\sum_{n=0}^{j-2}\sum_{l=0}^{n-1}(b_n^{j-1}+c_n^{j-1})(b_l^n+c_l^n a)H(a)^l.$$

One observes, by induction, that assertion (2.1) is correct. Now we must check the bounds on $b_{n}^j$ in terms of $n$ and $j.$ We will only prove the bound for $b_{n}^n$. 

Now, by the definition of $b_n^j$ and $c_n^j$ we see:

$$b_{n}^{j}=\frac{j}{j+1}\Big(b_{n-1}^{j-1}-\frac{1}{n}b_{n-1}^{j-1}-\sum_{l=n}^{j-2}(b_{l}^{j-1}+c_{l}^{j-1})b_{n}^l\Big)$$
$$c_{n}^{j}=\frac{j}{j+1}\Big(c_{n-1}^{j-1}-1 -\sum_{l=n}^{j-2}(b_l^{j-1}+c_l^{j-1})c_{n}^l \Big)$$

The fact that $$b_{j}^k=c_{j}^k=0, \,\, \text{when}\,\, k-j \,\, \text{is even}$$ can likely be seen from the recurrence relations. However, one can see it from the fact that $a$ is even with respect to $x=\frac{1}{2}$. Therefore, $H(a)$ is odd with respect to $x=\frac{1}{2}$.

In particular, in the expansion of $H(aH(a)^k)$ only the terms $H(a)^{k+1},$ $H(a)^{k-1},...$ should survive.

Assuming that the bounds held up at steps $j-1,...1$ we can bound 

$$|b^{j}_n| \leq\frac{j}{j+1}\Big( \frac{n-1}{n}(j-1)^{j-n} + \sum_{l=n}^{j-2} 2(j-1)^{j-l-1} l^{l-n}\delta_{n,l,j}\Big), $$

where $$\delta_{l,n,j}=0,\,\, \text{if}\,\, j-l \,\, \text{is odd OR} \,\, l-n \,\, \text{is even,}$$
$$\delta_{l,n,j}=1 \,\, \text{otherwise}.$$

In particular, we can remove the $l=n$ term.

Now we estimate:

$$\sum_{l=n+1}^{j-2} j^{j-l-1} l^{l-n}\leq \frac{1}{2}(j-n-1)j^{j-n-2}(n+1)$$

This is because there are at most $\frac{1}{2}(j-n-1)$ terms in the sum and all the terms are smaller than $j^{j-n-2}(n+1).$

Therefore, 

$$|b_n^j|\leq \frac{j}{j+1}\Big(\frac{n-1}{n}(j-1)^{j-n}+(j-n-1)(n+1)(j-1)^{j-n-2}  \Big)\leq \frac{j}{j+1}\Big(1+\frac{(j-n-1)(n+1)}{(j-1)^2}  \Big)(j-1)^{j-n}$$

$$=\frac{j}{j+1}\Big(1+\frac{(j-n-1)(n+1)}{(j-1)^2}  \Big)(1-\frac{1}{j})^{j-n}j^{j-n}.$$

In order to conclude, we need to prove that:

$$\frac{j}{j+1}\Big(1+\frac{(j-n-1)(n+1)}{(j-1)^2}  \Big)(1-\frac{1}{j})^{j-n}\leq 1.$$

We will prove that 

$$M:=\Big(1+\frac{(j-n-1)(n+1)}{(j-1)^2}  \Big)(1-\frac{1}{j})^{j-n}\leq 1.$$

First we call $A=\frac{j}{n}.$ Notice that since $j-n\geq 1$ we have that $An-n-1\geq 0.$ In fact, we may assume that $j-n\geq 2$ since when $j=n+1$ the inequality is trivial. 

Therefore, 

$$M=\Big(1+\frac{(An-n-1)(n+1)}{(An-1)^2}  \Big)(1-\frac{1}{An})^{(A-1)n}.$$

Now we have the following well known inequality: $$(1-\frac{1}{An})^{n}\leq e^{-\frac{1}{A}}.$$

This implies that

$$M\leq \Big(1+\frac{(An-n-1)(n+1)}{(An-1)^2}  \Big)e^{-\frac{A-1}{A}}.$$

But now $$e^{-\frac{A-1}{A}}\leq 1-\frac{A-1}{A}+\frac{1}{2}\Big(\frac{A-1}{A}\Big)^2=\frac{A^2+1}{2A^2}.$$

Therefore, $$M\leq (1+\frac{(An-n-1)(n+1)}{(An-1)^2}  \Big)\frac{A^2+1}{2A^2} $$ 
$$=\frac{(An-1)^2+((A-1)n-1)(n+1)}{(An-1)^2}\cdot\frac{A^2+1}{2A^2}=\frac{(An-1)^2+(A-1)n^2+(A-1)n-n-1}{(An-1)^2}\cdot\frac{A^2+1}{2A^2}$$

Now, $$M\leq 1$$
$$\iff$$

$$\Big((An-1)^2+(An-1-n)(n+1)\Big)(A^2+1)\leq 2A^2(An-1)^2$$

$$\iff$$

$$ (An-1-n)(n+1)(A^2+1)\leq(A^2-1)(An-1)^2$$

Now, since we can assume $j-n\geq 2$ we see that $An-1\geq n+1$. 

Therefore, the above is true 

$$\iff$$

$$(An-1-n)(A^2+1)\leq (A^2-1)(An-1)$$

$$\iff$$

$$n(A^2+1)\geq 2(An-1)$$

which is true for all $A\geq 1$ and all nonnegative integers $n$. 

Thus, $M\leq 1$ and so 

$$|b_n^j|\leq j^{j-n}.$$

By induction, this bound must hold for all $j$ and $n$. 

A similar proof can be used to bound the $c_{n}^j$ because we didn't actually use the $-\frac{1}{n}b^{j-1}_{n-1}$ term in the recursion formula for $b^{j}_n$ and the $-1$ in the formula for $c_n^j$ is negligible due to the presence of the $\frac{j}{j+1}$ term. 

\end{proof}

\emph{Remark:} The expansion above is not specific to the indicator function or the Hilbert transform; indeed, any $a$ and $H$ satisfying the following properties will do:

$(1) \, \, a^2= \lambda a, \lambda\in\mathbb{R}$

$(2) \, \,H^2=-1,$

$(3) \, \, H(fg)=gH(f)+fH(g) +H(H(f)H(g)).$

It is not clear to the author whether there are non-trivial operators other than the Hilbert transform which satisfy (2)-(3). However, one might imagine such a structure is possible in certain algebraic settings.

We have the following corollary:

\begin{cor}

Let $a$ be a function of one real variable satisfying that $a^2= \lambda a, \lambda\in\mathbb{R}$ (such as the characteristic function of a bounded set). Define $$\Omega_{a}:=\{\sum_{j=0}^\infty (\alpha_{j} +\beta_{j}a)H(a)^{j}\big | (\alpha_{j})_{j},(\beta_j)_j\in \ell^1, \alpha_{0}=0 \}.$$

Then, $H(\Omega_a)= \Omega_{a}.$ 

\end{cor}

A similar statement can be made if $a$ can be decomposed into a sum: $a=\sum_i \lambda_i a_i$ with $a_i^2=a_i$ by the linearity of $H.$

\begin{lemm}

Let $\big (\gamma_{k}(t)\big)_{k=-1}^\infty$ be a time dependent sequence of real numbers which solve the following system of ODE's:

$$\frac{d}{dt} \gamma_{k}(t)= \frac{1}{k}\gamma_{k-1}(t)+ \sum_{j\geq k+1} \gamma_{j}(t)d_{j,k}$$
$$\gamma_{-1}(t)\equiv 0$$
$$\gamma_{0}(t=0)=1,$$
$$\gamma_{k}(t=0)=0, k\geq 1.$$

Assume that $d_{j,k}$ satisfy the bound:

$$|d_{j,k}|\leq C j^{j-k}$$ for some fixed constant $C.$

Then, 

$$\gamma_{k}(t)\geq\frac{c}{(k!)^2}t^{k}$$ for all $t<\delta$ for some fixed $\delta.$

\end{lemm}

\begin{proof}

First we will truncate the system and derive a good a priori estimate on the solution of the truncated system.  First we will present the following a priori estimate:

\emph{Claim: Suppose that the coefficients $\gamma_{k}$ satisfy: }

$$\frac{t^k}{(k!)^2}\geq \gamma_{k}(t)\geq \frac{t^k}{100 (k!)^2}, \forall t\in [0,\delta],$$

for $\delta\leq\frac{1}{20\sqrt{C}}.$

Then, $$\frac{t^k}{(k!)^2}\geq\gamma_{k}(t)\geq \frac{t^k}{2(k!)^2} \forall t\in [0,\delta]. $$

\emph{Proof of the claim:}

$$\frac{d}{dt} \gamma_{k}(t)= \frac{1}{k}\gamma_{k-1}(t)+ \sum_{j\geq k+1} \gamma_{j}(t)d_{j,k}$$
$$= \gamma_{k-1}(t)\Big (\frac{1}{k}+ \sum_{j\geq k+1} \frac{1}{\gamma_{k-1}}\gamma_{j}(t)d_{j,k}\Big ) $$

By the assumption of the claim, $$\frac{1}{\gamma_{k-1}(t)}\gamma_{j}(t)\leq \frac{100(k-1)!^2}{j!^2}.$$

Now, by assumption $$|d_{j,k}|\leq C j^{j-k}.$$

Therefore,  $$\frac{1}{\gamma_{k-1}(t)}\gamma_{j}(t)d_{j,k}\leq \frac{100Ct^2(k-1)!^2}{j!^2}j^{j-k}$$

Now, for $j\geq k+1$ we can estimate:

$$\frac{1}{\gamma_{k-1}(t)}\gamma_{j}(t)d_{j,k}\leq \frac{100Ct^2(k-1)!^2}{j!^2}j^{j-k}\leq \frac{100Ct^2 (k-1)!^2}{(k+1)^k}\frac{j^j}{j!^2}$$

Now, $$\sum_{j\geq k+1}\frac{j^j}{(j!)^2}\leq \frac{k^{k}}{k!^2}.$$ This is due to the following fact:

$$|a_{k+1}|\leq \frac{1}{2}|a_k| \, \, \forall k \implies \sum_{n=k+1}^\infty |a_{n}|\leq |a_{k}|$$

Thus, 

 $$\sum_{j\geq k+1}\frac{1}{\gamma_{k-1}(t)}\gamma_{j}(t)d_{j,k}\leq \frac{100Ct^2(k-1)!^2}{(k+1)^k}\frac{k^k}{(k!)^2}\leq \frac{100Ct^2}{k^2}\leq\frac{1}{4k^2}$$ for $t\leq\delta\leq \frac{1}{20\sqrt{C}}.$

In particular, $$\frac{d}{dt}\gamma_{k}(t)\geq\frac{3\gamma_{k-1}(t)}{4k}$$

This completes the proof of the claim. 

Now, with the claim at hand we simply need to truncate the sequence and pass to a limit. We leave the details to the reader. 

\end{proof}

\begin{lemm}
Consider the function $$G(x,t)=\sum_{k=0}^\infty \frac{t^k \Big(Log\big|\frac{x}{x-1}\big|\Big)^k}{(k!)^2}.$$

Then, $$\Big(\int_{\frac{1}{2}}^1 G(x,t)^p dx\Big )^{1/p}\geq c e^{ctp}$$ for some fixed constant $c.$
\end{lemm}

\begin{proof}
Note that all of the terms in the series expansion of $G$ are positive on $[\frac{1}{2},1].$ Note further that 

$$(k!)^2\geq 4^{-k}(2k!).$$

Also recall the formula:

$$\sum_{k=0}^\infty \frac{x^k}{(2k!)}=\frac{1}{2}(e^{\sqrt{x}}+e^{-\sqrt{x}})=cosh(\sqrt{x}).$$
Thus, $$G(x,t)\geq cosh\Big (\sqrt{\frac{1}{4}t Log\big|\frac{x}{x-1}\big|}\Big)\geq \frac{1}{4} \exp\Big (\sqrt{\frac{1}{4}t Log\big|\frac{x}{x-1}\big|}\Big) $$ on $[\frac{1}{2},1].$ 

But then $$G(t,x)\geq\frac{1}{16}\exp \Big (\sqrt{-\frac{1}{4}t Log\big|{x-1}\big|}\Big) $$ since $x\in[\frac{1}{2},1].$ 

Now we compute:

$$\int_{\frac{1}{2}}^1 G(t,x)^p dx \geq \frac{1}{16}\int_{\frac{1}{2}}^1 \exp \Big (p\sqrt{-\frac{1}{4}t Log\big|{x-1}\big|}\Big)dx $$

Now change variables with $s=-Log|x-1|$ and we see:

$$\int_{\frac{1}{4}}^1 G(t,x)^p dx \geq\frac{1}{16}\int_{Log(2)}^\infty e^{-s}e^{p\sqrt{\frac{1}{4}t s}}ds$$

$$\geq e^{\frac{1}{4}tp^2} \frac{1}{16}\int_{Log(2)}^\infty e^{-s}e^{p\sqrt{\frac{1}{4}t s}}e^{-\frac{1}{4}tp^2}ds= e^{\frac{1}{4}tp^2} \frac{1}{16}\int_{Log(2)}^\infty e^{-(\sqrt{s}-\frac{\sqrt{t}}{2}p)^2} ds.$$

Now, we have that:

$$\inf_{a\geq 0}\int_{0}^{\infty} e^{-(\sqrt{s}-a)^2}ds\geq \frac{1}{2}.$$

Indeed,   $$\int_{0}^{\infty} e^{-(\sqrt{s}-a)^2}ds=2\int_{0}^{\infty} x e^{-(x-a)^2}dx=\int_{-a}^\infty(x+a)e^{-x^2}dx\geq\int_{0}^\infty(x+a)e^{-x^2}dx\geq \frac{1}{2}.$$ 

This completes the proof of the lemma.

\end{proof}

Now we are in a position to prove Theorem 1.1. 

\begin{proof} (Of Theorem 1.1)

We search for a solution of (1.1) of the following type:

\begin{equation} f(t,x)=\sum_{j} (\alpha_{j}(t)+\beta_{j}(t)a(x))H(a)^{j}(x,t). \end{equation}

Notice that since multiplication by $a$ and application of $H$ both keep $\Omega_a$ invariant, we see that $f$ must be of this form if the initial data is taken to be of this form.

We take $f_{0}=a$ so that only $\beta_{0}$ is non-zero initially. Now, we plug in our formula for $f$ and study the ODE's for the coefficients $\alpha_j$ and $\beta_j$.
Notice, first, that $$f(t,x)=\sum_{j} (\alpha_{j}(t)+\beta_{j}(t))H(a)^{j}(x,t), \forall x\in(0,1).$$

Using Lemma 2.4 it will suffice to prove that $$\alpha_{j}(t)+\beta_{j}(t)\geq\frac{ct^j}{(j!)^2}.$$

Now plug (2.2) into (1.1).

We see that $$\sum_j (\frac{d}{dt}\alpha_{j}(t)+\frac{d}{dt}\beta_j(t)a)H(a)^j=\sum_j (\alpha_{j}(t)+\beta_j(t))H(aH(a)^j).$$

Now we use Lemma 2.1 to simplify the RHS:

$$\sum_j (\alpha_{j}(t)+\beta_j(t))H(aH(a)^j)=\sum_j (\alpha_{j}(t)+\beta_j(t))\big(\frac{1}{j+1}(H(a)^{j+1})+\sum_{l=0}^{j-1}(b_{l}^{j}+c_{l}^j a)H(a)^l\big)$$ 

Now we match coefficients

$$\frac{d}{dt}\alpha_{j}(t)=\frac{1}{j-1}(\alpha_{j-1}+\beta_{j-1})+\sum_{l=j+1}^\infty (\alpha_{l}(t)+\beta_{l}(t))b^{l}_j$$
$$\frac{d}{dt}\beta{j}(t)=\sum_{l=j+1}^\infty (\alpha_{l}(t)+\beta_{l}(t))b^{l}_j$$

Now call $\alpha_j+\beta_j=\gamma_j.$

Then, 

$$\frac{d}{dt}\gamma_j(t)=\frac{1}{j-1}\gamma_{j-1}(t)+\sum_{l=j+1}^\infty\gamma_{j}(t)b_{j}^l.$$

Lemma 2.1 gives us the necessary bound on $b_j^l$ to apply Lemma 2.3 and we are done.

\end{proof}

\section{Application to a transport equation with divergence-free velocity field}

Consider the following equation in two dimensions:

$$\partial_{t} f+ \big(0,\chi_{[0,1]}\big)\cdot\nabla_{x,y} f=-H_{x}f,$$
 here $f$ is a function of $x$ and $y$, and $H_x$ is the Hilbert transform in the $x$ variable.

Then, we get:

$$\partial_{t} f+ \chi_{[0,1]}(x)\partial_y f=-H_{x}f.$$

Now take $f(x,y)=e^y g(x)$ as an ansatz. 

Then we get that $g$ must satisfy the following equation:

$$\partial_t g + \chi_{[0,1]}g=-H(g).$$

Thus, multiplying our equation by the integration factor $e^{t\chi_{[0,1]}(x)}$ we see:

$$\partial_{t} \big( e^{t\chi_{[0,1]}(x)} g(x)\big)=-e^{t\chi_{[0,1]}(x)}H(g)(x)$$

Now call $M(t,x)=e^{t\chi_{[0,1]}(x)} g(t,x).$

So, $$\partial_t M= -e^{t\chi_{[0,1]}(x)}H(e^{-t\chi_{[0,1]}(x)}M).$$

Now note that $$e^{t\chi_{[0,1]}(x)}=e^t \chi_{0,1}(x)+ (1-\chi_{[0,1]}(x))=(e^t-1)\chi_{[0,1]}(x)+1.$$

Then we see that $M$ satisfies the following equation:

$$\partial_{t}M=-\Big((e^t-1)\chi_{[0,1]}(x)+1\Big)H\Big(\big((e^{-t}-1)\chi_{[0,1]}(x)+1\big)M\Big)$$

Thus, $$\partial_{t}M=-H(M) -(e^t-1)\chi_{[0,1]}H(M)+ (1-e^{-t})H(\chi_{[0,1]}M)+ (e^t+e^{-t}-2)\chi_{[0,1]}H(\chi_{[0,1]}M).$$ Given the results of the previous section, it makes sense to search for a solution of the form:

$$M(t,x)=\sum_{k} (\alpha_{k}(t)+\beta_{k}(t)\chi_{[0,1]})Log\Big (\frac{x}{x-1} \Big)^k.$$

From the calculations in the previous section (Lemma 2.1), we see:

$$\frac{d}{dt}\alpha_k(t)=\frac{1}{\pi(k-1)}\beta_{k-1}-(1-e^{-t}) \frac{1}{\pi(k-1)}(\alpha_{k-1}+\beta_{k-1})+Error$$
$$\frac{d}{dt}\beta_k (t)=(e^{t}-1) \frac{1}{\pi(k-1)}\beta_{k-1}-(e^{t}+e^{-t}-2)\frac{1}{\pi(k-1)}(\alpha_{k-1}+\beta_{k-1})+Error$$

To clarify what will happen in this case we consider $t$ to be very small. Then we see (by replacing $e^t$ by $t+1$ since $t$ is small):

$$\frac{d}{dt} \alpha_{k}(t)=\frac{\beta_{k-1}}{\pi(k-1)}-\frac{t}{\pi(k-1)} (\alpha_{k-1}+\beta_{k-1})+Error$$
$$\frac{d}{dt} \beta_{k}(t)=\frac{t\beta_{k-1}}{\pi(k-1)}+ Error$$

Now, as before, we are going to take $\beta_0$ to be 1 initially and everything else to be 0 initially. 

Now we are going to solve the linear system:

$$\frac{d}{dt} \alpha_{k}(t)=\frac{\beta_{k-1}}{\pi(k-1)}-\frac{t}{\pi(k-1)} (\alpha_{k-1}+\beta_{k-1})$$
$$\frac{d}{dt} \beta_{k}(t)=\frac{t\beta_{k-1}}{\pi(k-1)}$$

We see that $$\beta_{k}(t)= \frac{(\frac{t^2}{2\pi})^{k}}{(k!)^2}.$$

Thinking that $\alpha_{k}$ and $\beta_{k}$ are positive we see that

$$\alpha_{k}(t)\leq \frac{t^{2k-1}}{(2k-1)(2\pi)^{k-1}(k-1)!^2}.$$

Now we want to put the Error back in and show that $$\beta_{k}(t)\geq \frac{1}{2}\frac{(\frac{t^2}{2\pi})^k}{(k!)^2}.$$

We get the following proposition:

\begin{prop}
Consider the following evolution equation:

$$\partial_{t}M=H(M) +(e^t-1)\chi_{[0,1]}H(M)- (1-e^{-t})H(\chi_{[0,1]}M)- (e^t+e^{-t}-2)\chi_{[0,1]}H(\chi_{[0,1]}M),$$
$$M(t=0,x)=\chi_{[0,1]}(x).$$

Then, for $t<c$, for some universal constant $c$, 

$$|M(t)|_{L^p}\geq e^{ct^2p}.$$

\end{prop}

\begin{cor}
Consider the evolution equation$$\partial_{t} f+ \big(0,\chi_{[0,1]}(x)\big)\cdot\nabla_{x,y} f=H_{x}f$$ with $$f(t=0,x,y)=e^{y}\chi_{[0,1]}(x).$$

Then, $$f(t,x,y)=e^y M(t,x)$$ with $$|M(t)|_{L^p}\geq e^{ct^2 p}.$$

\end{cor}

The proof of Proposition 3.1 is similar to the Proof of Theorem 1.1 except that we need the following variant of Lemma 2.3:

\begin{lemm}
Consider the following infinite system of ODE's:

$$\frac{d}{dt}\alpha_k(t)=\frac{1}{\pi(k-1)}\beta_{k-1}-(1-e^{-t}) \frac{1}{\pi(k-1)}(\alpha_{k-1}+\beta_{k-1})+\sum_{j\geq k+1}\alpha_{j}a^1_{jk}+\beta_{j}b^1_{jk}$$
$$\frac{d}{dt}\beta_k (t)=(e^{t}-1) \frac{1}{\pi(k-1)}\beta_{k-1}-(e^{t}+e^{-t}-2)\frac{1}{\pi(k-1)}(\alpha_{k-1}+\beta_{k-1})+\sum_{j\geq k+1}\alpha_{j}a^2_{jk}+\beta_{j}b^2_{jk}$$

with $$\alpha_{0}\equiv 0$$ and $$\frac{d}{dt}\beta_0=\sum_{j\geq 1}\alpha_{j}a_{0j}^2+\beta_{j}b_{0j}^2$$ and $|a_{jk}|,|b_{jk}|\leq 2j^{j-k}.$

Then, $$\beta_{k}\geq \frac{1}{50}\frac{t^{2k-1}}{(2k-1)(2\pi)^{k-1}(k-1)!^2}$$

\end{lemm}

The proof of the lemma is similar to the proof of Lemma 2.3 in that we just need to establish the following bootstrap estimate.

\emph{Bootstrap Estimate:}

Suppose that $\alpha_k$ and $\beta_k$ satisfy the following estimates:  $$100\frac{(\frac{t^2}{2\pi})^k}{k!^2}\geq \beta_{k}(t) \geq\frac{1}{100}\frac{(\frac{t^2}{2\pi})^k}{k!^2}$$ 
and $$200\frac{t^{2k-1}}{(2k-1)(2\pi)^{k-1}(k-1)!^2} \geq \alpha_{k}(t)\geq \frac{1}{200}\frac{t^{2k-1}}{(2k-1)(2\pi)^{k-1}(k-1)!^2}$$

on an interval of time $[0,\delta].$

\emph{Then,} a better estimate holds:

$$50\frac{(\frac{t^2}{2\pi})^k}{k!^2}\geq \beta_{k}(t) \geq\frac{1}{50}\frac{(\frac{t^2}{2\pi})^k}{k!^2}$$ 
and $$100\frac{t^{2k-1}}{(2k-1)(2\pi)^{k-1}(k-1)!^2} \geq \alpha_{k}(t)\geq \frac{1}{100}\frac{t^{2k-1}}{(2k-1)(2\pi)^{k-1}(k-1)!^2}$$ on $[0,\delta]$ so long as $\delta$ is small enough.

This is proven in exactly the same way as the claim in Lemma 2.3 was proven and we leave it to the reader.

Now that $\beta$ and $\alpha$ are of the order of $\frac{t^{2k}}{k!^2}$, Proposition 3.1 follows using Lemma 2.4.

\section{Acknowledgements}

The author thanks Professors P. Constantin and N. Masmoudi for fruitful discussions. He also acknowledges funding from an NSF Postdoctoral Fellowship.

\end{document}